\documentclass[12pt, reqno]{amsart}
\usepackage{amsmath, amsthm, amscd, amsfonts, amssymb, graphicx, color}
\usepackage[bookmarksnumbered, colorlinks, plainpages]{hyperref}
\hypersetup{colorlinks=true,linkcolor=red, anchorcolor=green, citecolor=cyan, urlcolor=red, filecolor=magenta, pdftoolbar=true}

\newtheorem{theorem}{Theorem}[section]
\newtheorem{lemma}[theorem]{Lemma}
\newtheorem{proposition}[theorem]{Proposition}
\newtheorem{corollary}[theorem]{Corollary}
\theoremstyle{definition}
\newtheorem{definition}[theorem]{Definition}

\theoremstyle{remark}
\newtheorem{remark}[theorem]{Remark}
\numberwithin{equation}{section}

\begin{document}
\setcounter{page}{1}

\title[CLUSTER SYSTEM]{ ON CLUSTER SYSTEMS OF TENSOR PRODUCT SYSTEMS OF HILBERT SPACES}

\author[Mukherjee, M.]{Mithun  Mukherjee}

\address{ Department of mathematics and statistics, IISER Kolkata, Mohanpur-741 252, India }
\email{\textcolor[rgb]{0.00,0.00,0.84}{ mithun.mukherjee@iiserkol.ac.in }}

%\dedicatory{This paper is dedicated to Professor ABCD}



\begin{abstract}
It is known that the spatial product of two product systems is intrinsic. Here we extend this result by analysing subsystems of the tensor product of product systems. A relation with cluster systems in the sense of \cite{BLM-root} is established. In a special case, we show that the amalgamated product of product systems through strictly contractive units is independent of the choices of the units. The amalgamated product in this case is isomorphic to the tensor product of the spatial product of the two and the type I product system of index one.
\end{abstract} \maketitle

\section{Introduction and preliminaries}

By a product system, we mean a measurable family of Hilbert spaces $(\mathcal E_t)_{t>0}$ with associative identification $\mathcal E_s\otimes \mathcal E_t=\mathcal E_{s+t}.$ Arveson \cite{Arv-continuous} associated with every $E_0$-semigroup, a product system of Hilbert spaces. He showed that $E_0$-semigroups are classified by their product systems up to cocycle conjugacy. Product systems are classified as spatial and non-spatial depending on whether or not there is a unit in the product system, where a unit is a measurable family of sections $(u_s)_{s>0},$ such that $u_s \in \mathcal E_s, s>0$
and $u_{s+t}=u_s\otimes u_t, s,t >0$ under the identification. The spatial product system has an index and the index is additive with respect to the tensor product of product systems. Much of the theory has a counterpart in the theory of product system of Hilbert modules (\cite{MuS-Markov}, \cite{BhS-tensor}). Though there is no natural tensor product in the category of product systems of Hilbert modules. To overcome this, Skeide (\cite{Ske-index}) introduced the notion of spatial product in the category of spatial product systems of Hilbert modules for which the index is additive with respect to the spatial product.

For two product systems $\mathcal E = (\mathcal E_t)_{t>0}$ and $\mathcal F = (\mathcal F_t)_{t>0}$ with reference units $u=(u_t)_{t>0}$ and $v=(v_t)_{t>0}$ respectively, their spatial product can be identified with the subsystem of the tensor product generated
by subsystems $(u_t\otimes \mathcal F_t)_{t>0}$ and $(\mathcal E_t \otimes v_t)_{t>0}$ . This raises another question, namely, whether the spatial product
is the tensor product or not. This has been answered in the negative sense by Powers \cite{Pow-addition}. This is exactly the same description of the product systems arising from the Powers sum of two $E_0$-semigroups. See \cite{Ske-commutant}, \cite{BLS-Powers}.

The spatial structure of a spatial product system depends on the reference unit. Indeed, Tsirelson (\cite{Tsi-automorphism}) showed that not all spatial product systems are transitive. i.e. there are spatial product systems in with two normalized units and without any automorphism of the product system sending one unit to another. This immediately raises the question whether different choice of reference units yields isomorphic product systems or not.   In \cite{BLMS-intrinsic}, it was shown in the affirmative sense. See also \cite{BLS-Powers}, \cite{Lie-tensor}. In other words, the spatial product of two spatial product systems is independent of the choice of the reference units.

In \cite{BhM-inclusion}, the amalgamated product of two product systems through contractive morphism is introduced. In a special case, when the contractive morphism is implemented by normalized units in respective spatial product systems, it is nothing but the spatial product of product systems of Hilbert spaces. The notion of amalgamation was motivated by the observation that the entire operation of obtaining a Powers sum can be obtained by a more general 'corner', in particular contractive (not necessarily isometric) intertwining semigroups.

 In this paper, we show the following : given two product systems $\mathcal E$ and $\mathcal F$ and their subsystems $\mathcal M$ and $\mathcal N$ respectively, the subsystem generated by  $\mathcal E\otimes \mathcal N$ and $\mathcal M\otimes \mathcal F$ is same as the subsystem generated by  $\mathcal E\otimes \check{\mathcal N}$ and $\check{\mathcal M} \otimes \mathcal F$  into $\mathcal E \otimes \mathcal F.$ Here $\check{\mathcal M}$ and $\check{\mathcal N}$ are respectively the cluster systems of $\mathcal M$ and $\mathcal N$ in the sense of \cite{BLM-root}. As a special case, we have the result of \cite{BLMS-intrinsic} namely spatial products of product systems are intrinsic. We also show that the amalgamated product of product systems through strictly contractive units is independent of the choices of the units and moreover the amalgamated product in this case is isomorphic to the tensor product of the spatial product of the two and the type I product system of index one.

\begin{remark}
It should be noted that some of these results also follow from the theory of random sets. See Proposition 5.3, \cite{Lie-tensor} for more details. See also Proposition 3.33, \cite{Lie-random} and the identification with the cluster construction given in Theorem 2.7, \cite{BLM-root}. But here we give a plain Hilbert space proof of this result.
\end{remark}

\section{Product systems and amalgamated product}

Let us start with some definitions.

\begin{definition}
A continuous tensor product system of Hilbert spaces (briefly: product system) is a measurable family $\mathcal E=(\mathcal E_t)_{t>0}$ of separable Hilbert spaces endowed with a measurable family of unitaries $V_{s,t}: \mathcal E_s\otimes \mathcal E_t\rightarrow \mathcal E_{s+t}$ for all $s,t >0,$ which fulfils for all $r,s,t > 0$  $$ V_{r,s+t}\circ (1_{\mathcal E_r} \otimes V_{s,t}) = V_{r+s,t} \circ (V_{r,s} \otimes 1_{\mathcal E_t} ). $$
\end{definition}

\begin{definition}
A unit $u$ of a product system is a measurable non-zero section $(u_t)_{t>0}$ through $(\mathcal E_t)_{t>0}$ which satisfies for all $s,t >0$ $$ V_{s,t}(u_s\otimes u_t)= u_{s+t}.$$ A unit is said to be normalized if $\|u_t\|=1$ for all $t>0.$
\end{definition}

\begin{definition}
Suppose $\mathcal E$ and $\mathcal F$ are product systems with associated unitaries $(V_{s,t})_{s,t>0}$ and $(W_{s,t})_{s,t>0}$ respectively. We say that $C=(C_t)_{t>0}$ is a contractive morphism from $\mathcal F$ to $\mathcal E$ if $(C_t)_{t>0}$ is a measurable family of contractions $C_t:\mathcal F_t \rightarrow \mathcal E_t$ with $$ C_{s+t}\circ W_{s,t} = V_{s,t}\circ (C_s\otimes C_t), ~ (s,t>0). $$ A contractive morphism is said to be an isomorphism if $C_t$ is a unitary for all $t>0.$
\end{definition}

\begin{definition}
A product system $\mathcal G$ with associated unitaries $U_{s,t}$ is said to be a product subsystem of $\mathcal E$ if $\mathcal F_t \subset \mathcal E_t$ for all $t>0$ and $ U_{s,t} = V_{s,t}|_{G_s \otimes G_t}$ for all $s,t>0.$
\end{definition}

\begin{remark}
We do not make the definition of measurability more explicit throughout this paper. For a thorough discussion, see Section 7, \cite{Lie-random}. In this context, we call by an \textbf{algebraic} product system is an object exactly like a product system but without having any measurable structure.
\end{remark}

Suppose $\mathcal E$ and $\mathcal F$ are two product
systems and $C=(C_t)_{t>0}:\mathcal F\rightarrow \mathcal E$ is a contractive morphism. Their amalgamated product $\mathcal E\otimes_C\mathcal F$ is introduced in \cite{BhM-inclusion} and can be seen equivalent to the following description (Theorem 2.7, \cite{Muk-index}) :  $\mathcal E\otimes_C\mathcal F$ is the unique algebraic product system with isometric morphisms of product systems $I:\mathcal E\rightarrow \mathcal E\otimes_C\mathcal F$ and $J:\mathcal F\rightarrow\mathcal E\otimes_C\mathcal F$ such that $$\mathcal E\otimes_C\mathcal F\simeq I(\mathcal E)\bigvee J(\mathcal F)$$ and for $x\in \mathcal E_t$ and $y\in \mathcal F_t,$ $$\langle I_t(x),J_{t}(y)\rangle = \langle x,C_ty\rangle,$$ where for product subsystems $\mathcal F$ and $\mathcal F^\prime$ of a product system $\mathcal E,$ let us denote by $\mathcal F \bigvee \mathcal F^\prime,$ the product system generated by $\mathcal F$ and $\mathcal F^\prime.$

Suppose $\mathcal E$ and $\mathcal F$ are two spatial product systems with normalized units $u$ and $v$ respectively. Define the contractive morphism $C=(C_t)_{t>0}:\mathcal F\rightarrow \mathcal E$ by $$C_t=e^{-\lambda t}|u_t\rangle\langle v_t|, ~\lambda \geq 0.$$ Let us denote by $\mathcal E\otimes_{(u,v,\lambda)}\mathcal F,$  the corresponding amalgamated product. For $\lambda=0,$ we simply denote it by $\mathcal E\otimes_{u,v}\mathcal F.$ We observe that $\mathcal E\otimes_{u,v}\mathcal F$ is the spatial product of $\mathcal E$ and $\mathcal F$ with respect to the reference units $u$ and $v$ respectively. Let $\Gamma$ be the type I product system of index $1.$ Choose and fix normalized units $\Omega$ and $\Omega^\prime$ of $\Gamma$ such that $\langle \Omega_t,\Omega^\prime_t\rangle=e^{-\lambda t}.$ These can be chosen easily as $\Gamma$ is isomorphic to the Fock product system $(\Gamma_{sym}(L^2[0,t]))_{t>0}$ and for the later, choose $\Omega$ to be the vacuum unit and $\Omega^\prime_t=e^{-\lambda t}\exp(\sqrt{2\lambda}\chi_{[0,t]}),$ for $t>0.$  Note that $\lambda=0$ if and only if $\Omega=\Omega^\prime.$ For any spatial product system $\mathcal E,$ we denote by $\mathcal E^I,$ the type I part of the product system, i.e. the smallest product subsystem of $\mathcal E$ containing all the units of $\mathcal E.$

\begin{proposition}\label{tensor}
Suppose $\mathcal E$ and $\mathcal F$ are two spatial product
systems. Let $u$ and $v$ be two normalized units of $\mathcal E$ and
$\mathcal F$ respectively. Then $\mathcal E\otimes_{(u,v,\lambda)}\mathcal F$
is isomorphic to the product system generated by $\mathcal E\otimes\Omega\otimes
v$ and $u\otimes\Omega^\prime\otimes\mathcal F$ inside $\mathcal E\otimes\Gamma\otimes\mathcal F,$
i.e. \begin{align*} \mathcal E\otimes_{(u,v,\lambda)}\mathcal F\simeq (\mathcal
E\otimes\Omega\otimes v)\bigvee (u\otimes\Omega^\prime\otimes\mathcal F).\end{align*}
\end{proposition}
\begin{proof}

 Define for each $t>0,$ $I_t:\mathcal E_t\rightarrow \mathcal E_t\otimes \Gamma_t\otimes \mathcal F_t$ and $J_t:\mathcal F_t \rightarrow \mathcal E_t \otimes \Gamma_t\otimes \mathcal F_t$ by, for $x\in \mathcal E_t,$ $I_t(x)=x\otimes \Omega_t\otimes v_t,$ for $y\in \mathcal F_t,$ $J_t(y)=u_t\otimes \Omega^\prime_t\otimes y.$ Then it is easy to see that $I=(I_t)_{t>0}$ and $J=(J_t)_{t>0}$ are isometric morphisms of product systems satisfying $\langle I_t(x),J_t(y)\rangle = \langle x, e^{-\lambda t}|u_t\rangle \langle v_t| y \rangle.$ Consequently from Theorem 2.7, \cite{Muk-index}, we get  $\mathcal E\otimes_{(u,v,\lambda)}\mathcal F \simeq I(\mathcal E)\bigvee J(\mathcal F),$ as algebraic product systems.  Now transferring the measurable structure of $ (\mathcal E\otimes \Omega \otimes v) \bigvee (u\otimes \Omega^\prime \otimes \mathcal F)$ onto $\mathcal E\otimes_{(u,v, \lambda)}\mathcal F $ via the isomorphism, we can make $\mathcal E\otimes_{(u,v,\lambda)}\mathcal F$ into a product system and the isomorphism becomes the isomorphism of product systems.

 \end{proof}

\section{Roots and Cluster Systems}

We denote the multiplication operation of the product system by $\circ$ i.e. $a\in \mathcal E_s,$ $b\in\mathcal E_t,$ we have $a\circ b\in \mathcal E_{s+t}.$  This notation is to differentiate the multiplication operation of the product system from the tensor product operation on the category of product systems. The following definition is adopted from \cite{BLM-root}.

\begin{definition} Let $\mathcal E$ be a spatial product system and let $u$ be a unit of this product system. A
measurable section $(a_t)_{t>0}$ of $\mathcal E$ is said to be a root of $u$ if
$$a_{s+t}=a_s\circ u_t+u_s\circ a_t~,~\langle a_t,u_t\rangle=0~,~\forall s,t>0.$$
\end{definition}
Note that for $t_1,t_2,\cdots,t_n>0$ with $\sum_{i=1}^nt_i=t,$ the following identity holds : $a_t=\sum_{i=1}^n y^i,$ where $y^i=u_{t_1}\circ u_{t_2}\circ\cdots\circ u_{t_{i-1}}\circ a_{t_i}\circ u_{t_{i+1}}\circ\cdots\circ u_{t_n}.$ Also note that $y^i$ and $y^j$ are orthogonal for $i\neq j.$ Hence $\|a_t\|^2=\sum_{i=1}^n\|y^i\|^2.$ Considering the symmetric Fock product system $\Gamma_{sym}(L^2[0,t],K),$ it is shown in Proposition 12, \cite{BLM-root} that the roots of the vacuum unit are given by $c\chi_{[0,t]},$ $c\in K.$ Note that the vacuum and $c\chi_{[0,t]},$ $c\in K$ generates the Fock product system and as a consequence we have the following result.

\begin{proposition}[Corollary 15, \cite{BLM-root}]\label{unit-roots}
Suppose $\mathcal E$ is a spatial product system and $u$ is a
unit. The product system generated by the unit $u$ and all roots of $u$ is $\mathcal E^I.$
\end{proposition}

Now we recall the notion of cluster system of a product system introduced in \cite{BLM-root}. Suppose $(\mathcal E,B)$ is a product system and $(\mathcal F,B|_F)$ is a product subsystem. Define $\tilde{\mathcal F_t}$ by $$\tilde{\mathcal F_t}=\overline{\mbox{span}}\{x\circ y: x \in \mathcal E_r\ominus \mathcal F_r, y \in \mathcal E_{t-r}\ominus \mathcal F_{t-r}, ~\mbox{for}~ \mbox{some}~ r,
0<r<t\}. $$ Set $\mathcal F^\prime_t=\mathcal E_t\ominus \tilde{\mathcal F_t}.$ Then $\mathcal F^\prime_{s+t}\subset \mathcal F^\prime_s \otimes \mathcal F^\prime_t$ for all $s,t > 0$ (\cite{BLM-root}). Denote by $\check{\mathcal F},$ the product subsystem generated by $\mathcal F^\prime.$ We call $\check{\mathcal F}$ the cluster of $\mathcal F.$ See \cite{BLM-root}. The name `cluster' comes from its connection to random sets (\cite{Lie-random}) which we now describe briefly. Suppose $\mathcal G$ is a product subsystem of a product system $\mathcal E.$ Then for every interval $[s,t],$ $0<s<t<1,$ we may identify, $\mathcal E_1\simeq \mathcal E_{s}\circ\mathcal E_{t-s}\circ\mathcal E_{1-t}.$   Let $P^\mathcal G_{s,t},$ $0\leq s\leq t\leq 1,$ be the family of commuting projections in $B(\mathcal E_1)$ defined by $$P^\mathcal G_{s,t}= P_{\mathcal E_{s}\circ\mathcal G_{t-s}\circ\mathcal E_{1-t}} = 1_{\mathcal E_s}\circ P_{\mathcal G_{t-s}}\circ
1_{\mathcal E_{1-t}},$$ where $P_K$ denotes the projection onto the subspace $K.$ Note that this operation is the usual tensor product of operators if the multiplication of the product system is viewed as the tensor product. From Proposition 3.18, \cite{Lie-random}, we know that $(s,t)
\rightarrow P^\mathcal G_{s,t}$ is jointly SOT continuous and the following holds : for every $x \in \mathcal E_1,$ $\|P^\mathcal G_{s,s+\epsilon}x-x\| \rightarrow 0$ and $\|P^\mathcal G_{t-\epsilon,t}x-x\|\rightarrow 0$ as $\epsilon \downarrow 0.$  So in the compact
simplex $\{0\leq s\leq t \leq 1\},$ it is uniformly continuous. i.e. for every $x \in \mathcal E_1,$
$\|P^\mathcal G_{s,t}x-x\|\rightarrow 0$ as $(t-s) \rightarrow 0.$
For $ n\geq 1,$ we have $$P^\mathcal G_{\frac{(i-1)}{n},\frac{i}{n}}=1_{\mathcal E_{\frac {1}{n}}}\circ
\cdots \circ 1_{\mathcal E_{\frac {1}{n}}}\circ P_{\mathcal G_{\frac {1}{n}}} \circ 1_{\mathcal E_{\frac {1}{n}}} \circ
\cdots \circ 1_{\mathcal E_{\frac {1}{n}}} ,$$ where $P_{\mathcal G_{\frac {1}{n}}}$ is on the i-th place.

Theorem 3.16, \cite{Lie-random} shows that any product subsystem $\mathcal G$ corresponds to a unique measure type $[\mu_\eta]$ ($\eta$ is a faithful state on $B(\mathcal E_1)$) on the closed subsets of $[0,1]$ such that the prescription $$ \chi_{\{Z:Z\cap [s,t]=\emptyset\}}\rightarrow
P^{\mathcal G}_{s,t},{~}((s,t)\in [0,1]) $$ extends to an injective
normal representation $J^{\mathcal G}_{\eta}$ of
$L^\infty(\mu^{\mathcal G}_\eta)$ on $\mathcal E_1.$   The mapping `cluster' which sends a closed set to its limit points is a measurable map on this space. More precisely, for any $Z \subset [0,1],$ denote $\check{Z}$ the set of its cluster points: $$\check{Z}=\{t\in Z:t \in \overline{ Z\setminus\{t\}}\}.$$  Then from Theorem 27, \cite{BLM-root}, we have \begin{align*} \label{check} J^{\mathcal G}_\eta(\chi_{\{Z:\check{Z}\cap [s,t]=\emptyset\}})=P^{\check{\mathcal G}}_{s,t},{~}((s,t)\in [0,1]).\end{align*}

\section{Subsystems of tensor product and their relation to cluster systems}

Our aim is to prove the following theorem.

\begin{theorem}\label{Main Theorem}
Suppose $\mathcal E$ and $\mathcal F$ are two product systems and also suppose $\mathcal M$ and $\mathcal N$ are product subsystems of $\mathcal E$ and $\mathcal F$ respectively. Then inside $\mathcal E\otimes \mathcal F,$  $$ \mathcal E \otimes \mathcal N \bigvee \mathcal M\otimes \mathcal F = \mathcal E \otimes \check{\mathcal N} \bigvee \check{\mathcal M} \otimes \mathcal F.$$
\end{theorem}

The  proof we postpone to the very end, after having illustrated the immediate consequences.

Let us define inductively $\mathcal M^{n+1}=\check{\mathcal M^n},$ where $\mathcal M^1=\check{\mathcal M}.$  Denote by $\mathcal M^\infty= \bigvee_n \mathcal M^n.$ Similarly for the subsystem $\mathcal N.$ Then we have the following corollary.

\begin{corollary}
If $\mathcal M^\infty = \mathcal E$ or $\mathcal N^\infty = \mathcal F,$ then $$ \mathcal E \otimes \mathcal N \bigvee \mathcal M\otimes \mathcal F = \mathcal E \otimes \mathcal F. $$
\end{corollary}

The case corresponding to $\lambda =0$ of the following corollary is the main result in \cite{BLMS-intrinsic}.

\begin{corollary}\label{independent}
Suppose $\mathcal E$ and $\mathcal F$ are two spatial product
systems with normalized units $u$ and $v$ respectively. Suppose $\lambda \geq 0.$ Then
$$ \mathcal E\otimes_{(u,v,\lambda)}\mathcal F \simeq  \left\{
 \begin{array}{cc}
 (\mathcal E\otimes
 \mathcal F^I) \bigvee (\mathcal E^I\otimes\mathcal F) & \mbox{if}~\lambda=0, \\
 (\mathcal E\otimes\Gamma\otimes
 \mathcal F^I) \bigvee (\mathcal E^I\otimes\Gamma\otimes\mathcal F) & \mbox{if}~\lambda > 0 \\
 \end{array}
 \right. $$
\end{corollary}

\begin{proof}

 Let $\Gamma$ be the type I product system of index one. Choose units $\Omega$ and $\Omega^\prime$ of $\Gamma$ such that for all $t>0,$ $\langle \Omega_t,\Omega^\prime_t \rangle = e^{-\lambda t}.$

\textbf{Case 1} : $\lambda=0.$ We get $\Omega=\Omega^\prime,$ which implies $\mathcal E\otimes_{(u,v)}\mathcal F \simeq (\mathcal E\otimes v) \bigvee (u \otimes\mathcal F).$ So it is enough to show $\mathcal E^I \subset \check{u}.$ For any root $a$ of $u,$ it is easy to see that $a \in \check{u}.$ Now the result follows from Proposition \ref{unit-roots}.

\textbf{Case 2} : $\lambda >0.$ We get $\Omega \bigvee \Omega = \Gamma$ as both sides have index one. From Proposition \ref{tensor}, we get $\mathcal E\otimes_{(u,v,\lambda)}\mathcal F \simeq (\mathcal
E\otimes\Omega\otimes v)\bigvee (u\otimes\Omega^\prime\otimes\mathcal F).$ First note that $$u\otimes \Gamma \otimes v \subset (\mathcal
E\otimes\Omega\otimes v)\bigvee (u\otimes\Omega^\prime\otimes\mathcal F).$$ Now from case 1 and the fact that $\check{\Omega}=\Gamma$ (as $\Gamma$ is of type I),  we have $$(u\otimes \Gamma \otimes v) \bigvee (\mathcal E \otimes \Omega \otimes v) = [(u \otimes \Gamma) \bigvee (\mathcal E \otimes \Omega)] \otimes v = \mathcal E \otimes \Gamma \otimes v.$$ It follows that $$\mathcal E \otimes \Gamma \otimes v \subset (\mathcal
E\otimes\Omega\otimes v)\bigvee (u\otimes\Omega^\prime\otimes\mathcal F).$$ Now again applying the result of case 1 for two product systems $\mathcal E \otimes \Gamma$ and $\mathcal F$ with respective units $u\otimes \Omega$ and $v,$ we get \begin{eqnarray*} ((\mathcal E \otimes \Gamma) \otimes v) \bigvee ((u \otimes \Omega) \otimes \mathcal F) &=& ((\mathcal E \otimes \Gamma) \otimes \mathcal F^I) \bigvee ((\mathcal E \otimes \Gamma)^I \otimes \mathcal F) \\ &=& (\mathcal E\otimes\Gamma\otimes
 \mathcal F^I) \bigvee (\mathcal E^I\otimes\Gamma\otimes\mathcal F) \\ &\simeq & (\mathcal E\otimes_{(u,v)}\mathcal F) \otimes \Gamma. \end{eqnarray*}

\end{proof}

The key of the proof of our main theorem is the following lemma.

\begin{lemma}\label{X_s}
Suppose $(\mathcal E,W)$ is a product system and $\mathcal F$ is a product subsystem of $(\mathcal E,W).$ Set  $\mathcal X_t=\mathcal F^\prime_t \ominus \mathcal F_t,$ $t>0.$ Then \begin{align*} \mathcal F_s\circ \mathcal X_t \oplus \mathcal X_s\circ \mathcal F_t = \mathcal X_{s+t}. \end{align*}
\end{lemma}

\begin{proof}

Suppose $x\in \mathcal X_t.$ consider the set
$$ A:=\{(z_1\circ z_2):z_1\in \mathcal E_r\ominus \mathcal F_r,~z_2\in \mathcal E_{s+t-r}\ominus \mathcal F_{s+t-r}~,~ \mbox{for}~\mbox{some}~ r,~ 0<r<s+t \}.$$
Then we claim that $\overline{\mbox{span}}~A=\overline{\mbox{span}}~(A_1\cup A_2\cup A_3),$
where \begin{align*} A_1=\{(y_1\circ y_2\circ y_3): y_1 \in \mathcal E_r\ominus \mathcal F_r, ~
y_2\in \mathcal E_{s-r},~ y_3 \in \mathcal E_t,~ \\ y_2\circ y_3 \in \mathcal E_{s+t-r}\ominus \mathcal F_{s+t-r},~
\mbox{for}~\mbox{some} ~0<r<s\},\end{align*}  \begin{align*} A_2=\{(y_1\circ y_2\circ
y_3): y_1 \in \mathcal E_s, y_2 \in \mathcal E_{r-s},~ y_1\circ y_2 \in \mathcal E_r \ominus \mathcal F_r, \\
y_3 \in \mathcal E_{s+t-r} \ominus \mathcal F_{s+t-r} ,~ \mbox{for}~\mbox{some} ~ s<r<s+t\} \end{align*} and $$ A_3=\{z_1\circ z_2: z_1 \in \mathcal E_s\ominus \mathcal F_s, z_2 \in \mathcal E_t \ominus \mathcal F_t\}.$$ Suppose $y_1\circ y_2\circ y_3 \in A_1.$ This implies $y_1 \in \mathcal E_r \ominus \mathcal F_r$ and $y_2\circ y_3 \in \mathcal E_{s+t-r}\ominus \mathcal F_{s+t-r}.$ This shows $y_1\circ y_2\circ y_3 \in A.$ We obtain $A_1\subset A.$ Similarly, $A_2, A_3 \subset A.$  We obtain, $\overline{\mbox{span}}~A\supset \overline{\mbox{span}}~(A_1\cup A_2 \cup A_3).$ For the converse, let $z_1\circ z_2\in A,$ with $z_1\in \mathcal E_r\ominus \mathcal F_r, z_2\in \mathcal E_{s+t-r}\ominus \mathcal F_{s+t-r}$ for some  $0<r<s.$ This implies $z_2 \in \overline{\mbox{span}} \{x_1\circ x_2: x_1\in \mathcal E_{s-r}, x_2\in \mathcal E_t, ~ x_1\circ x_2 \in \mathcal E_{s+t-r}\ominus \mathcal F_{s+t-r}\}.$ Clearly $z_1\circ x_1\circ x_2 \in A_1.$ We get $z_1\circ z_2 \in \overline{\mbox{span}}~A_1.$  Similarly, for $z_1\circ z_2 \in A$ with $z_1\in \mathcal E_r \ominus \mathcal F_r,$ $z_2\in \mathcal E_{s+t-r}\ominus \mathcal F_{s+t-r}$ for some $s<r<s+t,$ we have $z_1\circ z_2\subset \overline{\mbox{span}}~A_2.$  Therefore $\overline{\mbox{span}}~A\subset \overline{\mbox{span}}~(A_1\cup A_2\cup A_3).$ This proves the claim. Now let
$ y_1\circ y_2\circ y_3\in A_1$ be an arbitrary vector. Then
there is some $r_0,$ $0<r_0<s,$ such that $
y_1 \in \mathcal E_{r_0}\ominus \mathcal F_{r_0}, y_2 \in \mathcal E_{s-r_0},~ y_3 \in \mathcal E_t, y_2 \circ y_3 \in \mathcal E_{s+t-r_0} \ominus \mathcal F_{s+t-r_0}.$ Any vector $f_s \in \mathcal F_s$ is in the closed linear span of the vetors of the form $g^i_{r_0}\circ h^i_{s-r_0},$ where  $g^i_{r_0} \in \mathcal F_{r_0},$ $h^i_{s_0} \in \mathcal F_{s-r_0}.$  \allowdisplaybreaks{\begin{align*} \langle g^i_{r_0}\circ h^i_{s-r_0}\circ
x,y_1\circ y_2\circ y_3\rangle &= \sum_i\langle
g^i_{r_0},y_1\rangle\langle h^i_{s-r_0},y_2\rangle\langle x,y_3\rangle \\
&= 0.\end{align*}} %
This shows that $f_s\circ x \in {A_1}^\bot.$ Now let $y_1\circ y_2\circ y_3\in A_2$ be arbitrary. Then there is some $r_1,$ $s<r_1<s+t,$ such that $ y_1 \in \mathcal E_s, y_2 \in \mathcal E_{r_1-s},~ y_1\circ y_2 \in \mathcal E_{r_1} \ominus \mathcal F_{r_1} ,  y_3 \in \mathcal E_{s+t-r_1}\ominus \mathcal F_{s+t-r_1}.$ Now if
$y_1 \in \mathcal E_s \ominus \mathcal F_s,$ then for $f_s \in \mathcal F_s,$ the inner product $\langle
f_s\circ x,y_1\circ y_2\circ y_3\rangle=0$ and if $\langle
f_s,y_1\rangle\neq 0,$ then $y_2 \in \mathcal E_{r_1-s}\ominus \mathcal F_{r_1-s}$ and
$y_3\in \mathcal E_{s+t-r_1}\ominus \mathcal F_{s+t-r_1}.$ This is equivalent to
$y_2\circ y_3 \in \tilde{\mathcal F_t}.$ As $x\in \mathcal X_t \subset \mathcal F^\prime_t,$ the inner
product $\langle f_s\circ x,y_1\circ y_2\circ y_3\rangle=0.$ This shows $f_s \circ x \in {A_2}^\bot.$ For $z_1\otimes z_2\in A_3,$ it is easily seen that $\langle f_s\circ x, z_1\circ z_2 \rangle=0.$ Thus for arbitrary vector $z \in \overline{\mbox{span}}A,$ we have
$\langle f_s\circ x, z\rangle=0.$ Hence $f_s\circ x \in
\mathcal F^\prime_{s+t}.$ It is trivial to see that $f_s\circ x \in \mathcal E_{s+t}\ominus \mathcal F_{s+t}.$ We get $\mathcal F_s \circ \mathcal X_t \subset \mathcal X_{s+t}.$ Similarly $\mathcal X_s\circ \mathcal F_t\subset
\mathcal X_{s+t}.$

 On the other hand, we claim that
$$\mathcal X_{s+t} \supset \mathcal F_s\circ \mathcal X_t \oplus \mathcal X_s \circ \mathcal F_t.$$ Indeed, we observe that, the inclusion  $\mathcal F^\prime_{s+t}\subset \mathcal F^\prime_{s}\circ \mathcal F^\prime_t$ implies $$\mathcal X_{s+t}\subset \mathcal X_s \circ \mathcal F_t \oplus \mathcal F_s\circ \mathcal X_t \oplus \mathcal X_s\circ \mathcal X_t.$$ So it is enough to show that $\mathcal X_{s+t}\subset \mathcal E_{s+t}\ominus (\mathcal X_s\circ \mathcal X_t).$ But this follows from the fact that $\mathcal X_s\circ \mathcal X_t \subset \tilde{\mathcal F}_{s+t}.$

\end{proof}

\textbf{Proof of Theorem \ref{Main Theorem}} : It is enough to prove $ \mathcal E \otimes \check{\mathcal N} \subset \mathcal E \otimes \mathcal N \bigvee \mathcal M\otimes \mathcal F .$ By symmetry, the result follows. Fix the time point $t=1.$ It is enough to show that for $z\in \mathcal E_1,$
and  for $\eta\in \mathcal Y_1 := \mathcal N^\prime_1 \ominus \mathcal N_1,$ $z\otimes \eta \in
((\mathcal E\otimes \mathcal N)\bigvee (\mathcal M \otimes\mathcal F))_1.$ For
other time point, proof goes identically. Let $\epsilon > 0$ be
given. From uniform continuity
of $P^\mathcal M_{s,t},$ choose $N$ such that $n \geq N,$ $\|
z-P^\mathcal M_{\frac{i-1}{n},\frac{i}{n}}z\| \leq \frac{\epsilon}{\|\eta\|},$ for every $i=1,2,\cdots,n.$ From Lemma \ref{X_s}, the following decomposition holds : $\mathcal Y_1 = \oplus^n_{i=1} \mathcal Z_i ,$ where $\mathcal Z_i = \mathcal N_{\frac{1}{n}}\circ \mathcal N_{\frac{1}{n}}\circ \cdots \circ \mathcal Y_{\frac{1}{n}}\circ \cdots \circ \mathcal N_{\frac{1}{n}}$ with $\mathcal Y_{\frac{1}{n}}$ is on the $i$-th place.  Let $\eta=\oplus_i \eta_i$ be the corresponding (orthogonal) decomposition. Note that $\eta_i$  is in the closed linear span of elementary tensors of the form $P=p^1\circ p^2\circ\cdots \circ q\circ\cdots \circ p^n$ with $p^j\in N_{\frac{1}{n}}$ for $j\neq i$ and $q\in \mathcal Y_{\frac{1}{n}}.$ Also $P^\mathcal M_{\frac{i-1}{n},\frac{i}{n}}z$ is in the closed linear span of elementary tensors of the form $W=w^1\circ w^2\circ \cdots \circ v\circ \cdots \circ w^n$ with $w^j\in \mathcal E_{\frac{1}{n}}$ for $j\neq i$ and $v\in \mathcal M_{\frac{1}{n}}.$ Now note that \begin{eqnarray*}  W \otimes P &=&  (w^1\circ \cdots \circ v\circ \cdots \circ w^n) \otimes  (p^1\circ \cdots \circ q \circ \cdots \circ p^n)  \\ &=& (w^1 \otimes p^1) \circ \cdots \circ (v \otimes q) \circ \cdots \circ (w^n\otimes p^n)  \\ & \in & ((\mathcal E\otimes \mathcal N)\bigvee (\mathcal M \otimes\mathcal F))_1. \end{eqnarray*}  This gives us $P^\mathcal M_{\frac{i-1}{n},\frac{i}{n}}z \otimes \eta_i \in  ((\mathcal E\otimes \mathcal N)\bigvee (\mathcal M \otimes\mathcal F))_1. $ Now \begin{eqnarray*} \|z\otimes \eta - \sum^n_{i=1} P^\mathcal M_{\frac{i-1}{n},\frac{i}{n}}z \otimes \eta_i \|^2 &=& \sum^n_{i=1} \| (z-P^\mathcal M_{\frac{i-1}{n},\frac{i}{n}}z)\otimes \eta_i \|^2 \\  & < &  \sum^n_{i=1} \frac{\epsilon^2 \|\eta_i\|^2}{\|\eta \|^2} \\ & < & \epsilon^2 . \end{eqnarray*}   The result follows as the subspace is closed.  \qed

%---------------------------------------------------------------------------------------%

{\bf Acknowledgement.} I thank Professor B.V. Rajarama Bhat for several useful discussions on the subject. I also thank DST-Inspire (IFA-13 MA-20) for financial support.

\bibliographystyle{amsplain}

\end{document}